\documentclass[12pt]{amsart}

\usepackage{amsmath}
\usepackage{amssymb}
\usepackage{amsthm}
\usepackage[margin=1in]{geometry}
\usepackage{cleveref}

\theoremstyle{plain}
\newtheorem{theorem}{Theorem}[section]
\newtheorem{proposition}[theorem]{Proposition}
\newtheorem{lemma}[theorem]{Lemma}
\newtheorem{corollary}[theorem]{Corollary}

\theoremstyle{definition}
\newtheorem{definition}[theorem]{Definition}
\newtheorem{remark}[theorem]{Remark}

\newcommand{\Hn}{\mathbb{H}^n}
\newcommand{\op}{\mathrm{op}}

\title[Fourier transform on the Heisenberg Group] {Differentiation and Shift Property Characterisations \\
of the Group Fourier Transform on the Heisenberg Group}

\author{R. Lakshmi Lavanya}
\address{Department of Mathematics, Indian Institute of Science Education and Research (IISER) Tirupati, Yerpedu, Tirupati 517619, Andhra Pradesh, India.}
\email{rlakshmilavanya@labs.iisertirupati.ac.in}

\date{}

\begin{document}

\maketitle

\begin{abstract}
Recently, Krukowski characterised the Euclidean Fourier transform in terms of its behaviour with respect to differentiation. For locally compact abelian groups, he obtained a characterisation based on the behaviour of the Fourier transform under translation. Each of these results assumed that the operator (to be characterised) is an integral transform (with a kernel satisfying suitable regularity conditions), paired with a normalisation. In this paper, we obtain analogous results for the group Fourier transform on the Heisenberg group. Here, the differentiation property is formulated in terms of the behaviour of the group Fourier transform with respect to the left-invariant vector fields, and the shift property is based on its behaviour under right translation. Both of these characterise the group Fourier transform on the Heisenberg group, once paired with a normalisation. Our results assume that the transform is continuous linear, and do not require it to be an integral transform a priori.
\end{abstract}

\noindent{\small\textbf{2020 Mathematics Subject Classification.} Primary 43A80; Secondary 22E25, 22E27, 42A38, 43A30.}

\section{Introduction}

It is well-known that the Euclidean Fourier transform maps the convolution product into the pointwise product. Over the past two decades, several characterisations of the Fourier transform have been obtained based on this property in addition to other properties. In 2008, Alesker, Artstein-Avidan, and Milman \cite{AAM}, and in 2010, Jaming  \cite{Jaming} obtained such characterisations. This was followed by analogous characterisations of the Fourier transform for the Heisenberg group (\cite{LT2012},\cite{LT2014}), compact groups (\cite{KS}, \cite{Krukowski2}) and locally compact abelian groups (\cite{Krukowski1}). Recently, Achar, Biswas and Manna proved similar characterisations of the Euclidean Fourier transform (\cite{ABM1}), and the group Fourier transform on the Heisenberg group (\cite{ABM2}), both for vector-valued functions. 

In 2024, Krukowski (\cite{Krukowski2}) formulated characterisations of the Fourier transform, based on its behaviour with respect to shift operations (on locally compact abelian groups) and differentiation (for the Euclidean space). Unlike the convolution property, the shift- and the differentiation-based hypotheses involve only one function. This restriction is overcome by an additional assumption, which Krukowski refers to as the 'Dirac Delta property' (see \Cref{def:diracseq,def:diracfam} below). Unlike the Euclidean space, a general locally compact abelian group does not carry an inherent differentiable structure. The Heisenberg group presents itself as an interesting example to study, as it is non-abelian, and has a rich differentiable structure. Recently, such results were obtained for the Strichartz Fourier transform on the Heisenberg group (\cite{Ksft}). In this paper, we prove analogues of Krukowski's results for the group Fourier transform on the Heisenberg group.

First we state Krukowski's definitions and results. 

\begin{definition}[\cite{Krukowski2}]\label{def:diracseq}
We say that $(\delta_k) \subset C_c^1(\mathbb{R})$ is a \textit{Dirac delta sequence} if every $\delta_k$
\begin{enumerate}
    \item[(i)]  has compact support contained in $\left[-\dfrac{1}{k}, \dfrac{1}{k}\right]$,
    \item[(ii)] is nonnegative and symmetric, i.e., $\delta_k(x) = \delta_k(-x)$ for every $x \in \mathbb{R}$,
    \item[(iii)] satisfies the equality
    \[
    \int_{\mathbb{R}} \delta_k(x) \, dx = 1.
    \]
\end{enumerate}
\end{definition}

\begin{definition}[\cite{Krukowski2}]
We say that $T : L^1(\mathbb{R}) \to C^b(\mathbb{R})$ satisfies the \textit{Dirac delta property} if there exists a Dirac delta sequence $(\delta_k) \subset C_c^1(\mathbb{R})$ such that for every $y \in \mathbb{R}$ we have
\begin{equation}
    \lim_{k \to \infty} T(\delta_k)(y) = 1.
\end{equation}
\end{definition}

Krukowski obtained the following characterisation of the Euclidean Fourier transform, in terms of its behaviour with respect to differentiation.

\begin{definition}[\cite{Krukowski2}] Assume $T$ is an integral transform $T : L^1(\mathbb{R}) \to C^b(\mathbb{R})$, such that
\begin{equation*}
T(f)(y) := \int_{\mathbb{R}} K(x,y) \, f(x) \, dx,
\end{equation*}
where $K \in C^b(\mathbb{R} \times \mathbb{R})$ is differentiable with respect to the first variable. Such a map $T$ is said to satisfy the \textit{time differentiation property,} if
\begin{equation}
T(f')(y) = 2\pi i y \, T(f)(y)  \label{eq:kdiffn}
\end{equation}
for every $f \in C_c^1(\mathbb{R})$ and $y \in \mathbb{R}$. 
\end{definition}

\begin{theorem}[{\cite{Krukowski2}}]
\label{thm:kruk}
Suppose that $T : L^1(\mathbb{R}) \to C^b(\mathbb{R})$ is an integral transform with kernel $K \in C^b(\mathbb{R} \times \mathbb{R})$, which is differentiable with respect to the first variable. If $T$ satisfies
\begin{itemize}
    \item the time differentiation property, and
    \item the Dirac delta property,
\end{itemize}
then $T$ is the Fourier transform.
\end{theorem}

Now we state Krukowski's definitions and characterisation related to the shift-based property.

Let $G$ be a locally compact abelian group with dual group $\widehat{G}$.  Let $N_0$ denote the set of all open, nonempty neighborhoods of the identity element of $G.$

\begin{definition}[\cite{Krukowski2}] \label{def:diracfam}
A family $(\delta_U)_{U \in N_0}$ is called a \textit{Dirac delta family,} if every $\delta_U$
\begin{itemize}
    \item is continuous and compactly supported, i.e., $\delta_U \in C_c(G)$,
    \item is nonnegative and symmetric, i.e., $\delta_U(x) = \delta_U(-x)$ for every $x \in G$,
    \item satisfies the equality $\int_G \delta_U(x) \, dx = 1$.
\end{itemize}
\end{definition}

\begin{definition}[\cite{Krukowski2}]
We say that $T : L^1(G) \to C^b(\widehat{G})$ satisfies the \textit{Dirac delta property} if there exists a Dirac delta family $(\delta_U)$ such that for every $\chi \in \widehat{G}$ we have $$\lim_{U \to 0} T(\delta_U)(\chi) = 1,$$
or, more explicitly,
\[
\forall_{\varepsilon > 0} \ \exists_{U \in N_0} \ \forall_{\substack{V \subset U \\ V \in N_0}} \ |T(\delta_V)(\chi) - 1| < \varepsilon.
\]
\end{definition}

\begin{theorem}[{\cite{Krukowski2}}]
If a linear and bounded operator $T : L^1(G) \to C^b(\widehat{G})$ satisfies
\begin{itemize}
    \item the time shift property, i.e., for every $f \in L^1(G)$ and $x \in G$, $\chi \in \widehat{G}$ we have
    \begin{equation}
T(L_x f)(\chi) = \overline{\chi(x)} \, T(f)(\chi),\label{eq:ktransln}
\end{equation}
    \item the Dirac delta property,
\end{itemize}
then $T$ is the Fourier transform.
\end{theorem}

In this paper, we obtain differentiation- and shift-based characterisations for the (group) Fourier transform on the Heisenberg group. The proofs of our results are more complicated than those of the above results, due to the following reasons:
\begin{itemize}
\item For any locally compact abelian group (including $\mathbb{R}$), the Fourier transform of a function is scalar-valued, whereas in the case of the Heisenberg group, the group Fourier transform is operator-valued.

\item Both the results stated above assume that the operator is an integral transform, with a kernel satisfying some regularity condition. The results for the Heisenberg group assume that the operator is continuous and linear, and derive that it is an integral transform whose kernel satisfies a regularity condition.
\end{itemize}

Despite the above-mentioned challenges, the Heisenberg group presents a differentiable structure. This simplifies the hypothesis of the shift-based characterisation to finitely many points unlike the infinitely many points in \Cref{eq:ktransln}. Also, the hypothesis of the differentiation-based characterisation involves only 2n relations in terms of the left-invariant vector fields, unlike the a priori expectation of $(2n+1)$, as in \Cref{eq:kdiffn}. We revisit this discussion with more technical details in Remarks \ref{rem:shift1}, \ref{rem:shift2} and \ref{rem:diffn}.\\

This paper is organised as follows: Section~\ref{sec:prelims} introduces preliminary definitions and notations related to the Heisenberg group $\Hn$. Section~\ref{sec:ddhn} establishes the analogues of Dirac delta family and Dirac delta property in the context of the Heisenberg group. Also, it derives a kernel representation for $B$ (\Cref{lemE}) from continuity and linearity of $B$ alone. This idea underlies both of the main theorems. Section~\ref{sec:diff} formalises a differentiation property in terms of the left-invariant vector fields on $\Hn$, and shows that this property, when paired with a Heisenberg Dirac delta property, characterises the group Fourier transform. Section~\ref{sec:shift} proves the analogous statement for a shift property, based on right translation. In both cases the map is assumed only to be continuous and linear. Finally, Section~\ref{sec:furtherdisc} shows that the differentiation and shift hypotheses are in fact equivalent, and discusses the differences in the nature of the hypotheses used in the earlier settings ({\cite{Krukowski2}) and the current one.

\section{Preliminaries}
\label{sec:prelims}
In this section, we recall definitions and fix our notations regarding the Heisenberg group setting.

\subsection{The Heisenberg group and its group Fourier transform}

We recall the definitions of the group law and the Fourier transform on the Heisenberg group.\\

Let $\Hn=\mathbb{C}^n\times\mathbb{R}$ denote the $(2n+1)$-dimensional Heisenberg group with group law
\[
(z,t)(w,s)=\Big(z+w,\ t+s+\tfrac12\operatorname{Im}(z\cdot\overline{w})\Big),\qquad (z,t),(w,s)\in\Hn.
\]
With respect to this group law, the centre of $\Hn$ is the set $\{(0,t): t \in \mathbb{R}\},$ which can be identified with $\mathbb{R}.$ The Haar measure on $\Hn$ is the Lebesgue measure $dz\,dt$ on $\mathbb{C}^n\times\mathbb{R}.$ Since $\Hn$ is unimodular, this measure is both left- and right-invariant. By the Stone-von Neumann theorem (see \cite{Thangavelu} or \cite{Folland}), every infinite-dimensional irreducible unitary representation of $\Hn$ acting on $L^2(\mathbb{R}^n),$ which is nontrivial at the centre, is unitarily equivalent to one of the Schr\"odinger representations $\pi_\lambda$, $\lambda\in\mathbb{R}^*:=\mathbb{R}\setminus\{0\}$, which are given by
\[
\pi_\lambda(z,t)\varphi(\xi)=e^{i\lambda t}e^{i\lambda(x\cdot\xi+\frac12 x\cdot y)}\varphi(\xi+y),\qquad \varphi\in L^2(\mathbb{R}^n),\ \xi\in\mathbb{R}^n,
\]
where $z=x+iy\in\mathbb{C}^n.$ Write $\pi_\lambda(z):=\pi_\lambda(z,0),$ so that $\pi_\lambda(z,t)=e^{i\lambda t}\pi_\lambda(z)$. For $f\in L^1(\Hn),$ its group Fourier transform  is the operator on $L^2(\mathbb{R}^n),$ defined as
\[
\pi_\lambda(f)=\int_{\Hn}f(z,t)\,\pi_\lambda(z,t)\,dz\,dt.
\]
The group Fourier transform of $f$ at $\lambda$ is a bounded operator on $L^2(\mathbb{R}^n)$ with $\|\pi_\lambda(f)\|_{\op}\le\|f\|_1$. 

Throughout this paper, $(\cdot,\cdot)$ denotes the inner product on $L^2(\mathbb{R}^n),$ linear in the first argument and conjugate-linear in the second.

\subsection{Multiplication and differentiation operators on $L^2(\mathbb{R}^n)$}

This subsection gives the relation between the multiplication and differentiation operators, which combined with the next section, will illustrate the relation between the two main results.\\
 
For $j=1,\dots,n$ let $$Q_j\varphi(\xi)=\xi_j\varphi(\xi) \ \textrm{and \ } P_j\varphi(\xi)=\dfrac{\partial\varphi}{\partial\xi_j}(\xi),$$ for $\varphi$ in the Schwartz space $\mathcal{S}(\mathbb{R}^n)$. The Schwartz space  is invariant under $Q_j$ and $P_j$, for all $j=1,\dots,n$.


Then $Q_j$ is self-adjoint and $P_j^*=-P_j$, and, for all $j=1,\dots,n.$ Moreover, we have the following commutation relations for all $j,k=1,\dots,n$.
\begin{equation}
\label{CCR}
[Q_j,P_k]:=Q_jP_k-P_kQ_j=-\delta_{jk}\,I,\qquad [Q_j,Q_k]=0,\qquad [P_j,P_k]=0 .
\end{equation}
Throughout this paper, all operator identities involving $Q_j$ and $P_j$ are stated as acting on functions in $\mathcal{S}(\mathbb{R}^n)$, and extended as bounded operators on all of $L^2(\mathbb{R}^n)$ by continuity, whenever the operators are bounded.
The actions of the one-parameter unitary groups generated by $P_j,Q_j$ are given by
\[
e^{isQ_j}\varphi(\xi)=e^{is\xi_j}\varphi(\xi),\qquad e^{sP_j}\varphi(\xi)=\varphi(\xi+se_j),\qquad s\in\mathbb{R},
\]
where $e_j\in\mathbb{R}^n$ is the $j$-th standard basis vector.

\subsection{Left-invariant vector fields on $\Hn$}
Here, we recall the definition of the left-invariant vector fields on $\Hn,$ and discuss some of their properties.\\
 
Write $z=x+iy\in\mathbb{C}^n$. The left-invariant vector fields on $\Hn$ are given by
\[
X_j=\frac{\partial}{\partial x_j}+\frac12 y_j\frac{\partial}{\partial t},\qquad
Y_j=\frac{\partial}{\partial y_j}-\frac12 x_j\frac{\partial}{\partial t},\qquad j=1,\dots,n.
\]

Each is characterised, at any $g\in\Hn$, by differentiation along the one-parameter subgroup it generates, given as
\[
X_jf(g)=\frac{d}{dr}\Big|_{r=0}f\big(g\cdot(re_j,0)\big),\qquad
Y_jf(g)=\frac{d}{dr}\Big|_{r=0}f\big(g\cdot(ire_j,0)\big).
\]

A left-invariant vector field commutes with left translation in the following sense. For fixed $h\in\Hn$, writing $L_hf(g):=f(h^{-1}g)$,
\begin{equation}
\label{leftcomm}
X_j(L_hf)(g)=(X_jf)(h^{-1}g),\qquad Y_j(L_hf)(g)=(Y_jf)(h^{-1}g),
\end{equation}
since $h^{-1}(g\cdot(re_j,0))=(h^{-1}g)\cdot(re_j,0)$ by associativity, and similarly for $(ire_j,0)$.

A direct computation gives the Heisenberg bracket relations
\begin{equation}
\label{bracket}
[X_j,Y_j]=-\frac{\partial}{\partial t},\qquad [X_j,Y_k]=0\ (j\ne k),\qquad [X_j,X_k]=[Y_j,Y_k]=0,
\end{equation}
for all $j,k=1,\dots,n$.

Since $X_j,Y_j$ are divergence-free with respect to $dz\,dt$ and $\Hn$ is unimodular, integration by parts gives
\begin{equation}
\label{ibp}
\int_{\Hn}(X_jf)\,h\,dz\,dt=-\int_{\Hn}f\,(X_jh)\,dz\,dt,\qquad
\int_{\Hn}(Y_jf)\,h\,dz\,dt=-\int_{\Hn}f\,(Y_jh)\,dz\,dt,
\end{equation}
whenever $f,h\in C_c^\infty(\Hn)$, or $f$ is Schwartz and $h$ is smooth with all derivatives of polynomial growth (see e.g.\ \cite{Thangavelu}).

For $f\in L^1_{\mathrm{loc}}(\Hn)$ the distributions $X_jf$ and $Y_jf$ are defined by requiring
\begin{equation}
\label{ibp-ext}
\int_{\Hn}f\,(X_jh)\,dz\,dt=-\int_{\Hn}(X_jf)\,h\,dz\,dt,\qquad
\int_{\Hn}f\,(Y_jh)\,dz\,dt=-\int_{\Hn}(Y_jf)\,h\,dz\,dt
\end{equation}
for every $h\in C_c^\infty(\Hn)$, agreeing with the classical derivative by \eqref{ibp} when $f$ is $C^1$. \\

We denote by $\mathcal{S}(\Hn)$ the Schwartz class on $\Hn\cong\mathbb{R}^{2n+1}$. Its basic properties with respect to the group Fourier transform were established in \cite{Geller}. The group Fourier transform satisfies the following property with respect to the left-invariant vector fields on the Heisenberg group. \\

The following result is implicit in \cite{Geller}. We state it here for completeness.
\begin{lemma}
\label{lemA}
For every $\lambda\in\mathbb{R}^*$, $(z,t)\in\Hn$, $\varphi\in\mathcal{S}(\mathbb{R}^n)$ and $j=1,\dots,n$, we have
\[
X_j\pi_\lambda(z,t)\varphi=i\lambda\,\pi_\lambda(z,t)Q_j\varphi,\qquad Y_j\pi_\lambda(z,t)\varphi=\pi_\lambda(z,t)P_j\varphi.
\]
\end{lemma}

\begin{proposition}[{\cite{LT2014}}]
\label{prop:lt23}
For $f\in\mathcal{S}(\Hn)$, $\lambda\in\mathbb{R}^*$, and $j=1,\dots,n$, the following identities hold as identities of operators on $\mathcal{S}(\mathbb{R}^n)$:
\begin{equation}
\pi_\lambda(iX_jf)=\lambda\,\pi_\lambda(f)\,Q_j, \qquad
\pi_\lambda(-Y_jf)=\pi_\lambda(f)\,P_j.
\end{equation}
\end{proposition}

\subsection{Right translation}

In this subsection, we discuss the right translation on the Heisenberg group and its behaviour under the group Fourier transform.

\begin{definition}
\label{def:righttrans}
Let $(z,t)=(x+iy,t)\in\Hn$, with $x,y\in\mathbb{R}^n$. For a function $f$ on $\Hn$, its right translate by $(z,t)\in\Hn$, denoted $R_{(z,t)}f$, is defined as 
\[
(R_{(z,t)}f)(w,s):=f\big((w,s)(z,t)\big),\qquad (w,s)\in\Hn.
\]
\end{definition}

For $j=1,\dots,n$, let $e_j\in\mathbb{R}^n$ be the $j$-th standard basis vector. For $s\in\mathbb{R}$, write $(se_j,0)$ for the element of $\Hn$ with $x=se_j$, $y=0$, $t=0$, and $(ise_j,0)$ for the element with $x=0$, $y=se_j$, $t=0$. Thus $R_{(se_j,0)}$ and $R_{(ise_j,0)}$ are the corresponding instances of the operator of Definition~\ref{def:righttrans}, and $s\mapsto(se_j,0)$, $s\mapsto(ise_j,0)$ are the one-parameter subgroups of $\Hn$ generated by $X_j$ and $Y_j$ respectively. A direct computation using the group law gives
\begin{equation}
\label{pivals}
\pi_\lambda(se_j,0)=e^{is\lambda Q_j},\qquad \pi_\lambda(ise_j,0)=e^{sP_j}.
\end{equation}

The group Fourier transform satisfies the following property.

\begin{proposition}[Property~3, {\cite{LT2012}}]
\label{prop:lt-prop3}
For $f\in L^1(\Hn)$ and $(z,t)\in\Hn$,
\[
\pi_\lambda(R_{(z,t)}f)=\pi_\lambda(f) \ \pi_\lambda(z,t)^*,\qquad \lambda\in\mathbb{R}^*.
\]
\end{proposition}

\section{Setting the Heisenberg group analogues}
\label{sec:ddhn}
In this section, we build the background in the Heisenberg group setting which is required to state our main results.

\subsection{The Dirac delta property}

The following are adaptations of the Dirac delta family and Dirac delta property of \cite{Krukowski2} to $\Hn$.

\begin{definition}[Heisenberg Dirac delta family]
\label{def:ddf}
Let $N_0$ denote the collection of open neighbourhoods of the identity $(0,0) \in \Hn$. A family $(\delta_U)_{U \in N_0} \subset C_c(\Hn)$ is called a \emph{Heisenberg Dirac delta family} if every $\delta_U$
\begin{enumerate}
    \item[(i)] has support contained in $U$,
    \item[(ii)] is nonnegative,
    \item[(iii)] satisfies the equality
    \[
    \int_{\Hn} \delta_U(z,t) \, dz\, dt = 1.
    \]
\end{enumerate}
\end{definition}

Note that the above definition differs from the Dirac delta sequences and families of \cite{Krukowski2} as no symmetry ($\delta_U(x^{-1})=\delta_U(x)$) is imposed. This relaxation does not affect the arguments given below.

\begin{definition}[Heisenberg Dirac delta property]
\label{def:ddp}
A linear map $B$ assigning to each $f\in L^1(\Hn)$ an operator-valued function $\lambda\mapsto B(f)(\lambda)\in\mathcal{B}(L^2(\mathbb{R}^n))$ satisfies the \emph{Heisenberg Dirac delta property} if there is a Heisenberg Dirac delta family $(\delta_U)_{U\in N_0}$ on $\Hn$ such that, for a.e.\ $\lambda\in\mathbb{R}^*$,
\[
\lim_{U\to(0,0)}B(\delta_U)(\lambda)=I
\]
strongly on $L^2(\mathbb{R}^n)$.
\end{definition}

We now study the behaviour of a Heisenberg Dirac delta family under the group Fourier transform.

\begin{proposition}
\label{prop:pi-ddp}
The group Fourier transform satisfies the Heisenberg Dirac delta property with respect to any Heisenberg Dirac delta family.
\end{proposition}

\begin{proof}
For any Heisenberg Dirac delta family $(\delta_U)$, strong continuity of $\pi_\lambda$ together with $\int_{\Hn}\delta_U=1$ and $\delta_U\ge0$ gives 
$$\pi_\lambda(\delta_U)=\int_{\Hn}\delta_U(z,t)\pi_\lambda(z,t)\,dz\,dt\to\pi_\lambda(0,0)=I,$$
strongly as $U\to(0,0)$. Indeed, for any $\varphi\in L^2(\mathbb{R}^n)$, we have
\[
\|\pi_\lambda(\delta_U)\varphi-\varphi\|_2\le\int_{\Hn}\delta_U(g)\,\|\pi_\lambda(g)\varphi-\varphi\|_2\,dg\le\sup_{g\in U}\|\pi_\lambda(g)\varphi-\varphi\|_2 ,
\]
which tends to $0$ as $U\to(0,0)$.
\end{proof}

\subsection{Continuous linear maps on $L^1(\Hn)$ and their kernels}

\begin{definition}[Setting]
\label{def:clt}
Throughout, $B$ denotes a continuous linear map
\[
B:L^1(\Hn)\longrightarrow L^\infty\big(\mathbb{R}^*,\mathcal{B}(L^2(\mathbb{R}^n)),d\mu\big),\qquad d\mu(\lambda)=(2\pi)^{-n-1}|\lambda|^n\,d\lambda.
\]
\end{definition}
Only the null sets of $d\mu$ play a role below, and these are the null sets of Lebesgue measure on $\mathbb{R}^*$. The space $L^\infty(\mathbb{R}^*,\mathcal{B}(L^2(\mathbb{R}^n)),d\mu)$ is understood in the weak sense, that is, $\lambda\mapsto(B(f)(\lambda)\varphi,\psi)$ is measurable for all $\varphi,\psi\in L^2(\mathbb{R}^n)$. Measurability in $\lambda$ is not used in any argument below.

The following discussion lays the foundation for our understanding and formulation of the main results.

\begin{remark}[Choice of representatives]\label{rem:reps}
For each $f\in L^1(\Hn)$, the function $B(f)(\cdot)$ is a priori defined only almost everywhere,
with an exceptional null set depending on $f$. We normalise this in the standard way: fixing a
countable $\mathbb{Q}[i]$-linear subspace $D_0\subset L^1(\Hn)$, dense in $L^1(\Hn)$, we pass
to the bounded extension of $f\mapsto B(f)(\lambda)$ from $D_0$ to $L^1(\Hn)$. This gives a set
$\Lambda\subset\mathbb{R}^*$, with $\mathbb{R}^*\setminus\Lambda$ of Lebesgue measure zero, and a
value $B(f)(\lambda)$ defined for \emph{every} $f\in L^1(\Hn)$ and every $\lambda\in\Lambda$, satisfying the following:
\begin{enumerate}
\item[(i)] the map $f\mapsto B(f)(\lambda)$ is complex-linear and $\|B(f)(\lambda)\|_{\op}\le C\|f\|_1$,
for every $\lambda\in\Lambda$,
\item[(ii)] for every $f\in L^1(\Hn)$, the normalised value $B(f)(\lambda)$ agrees with the
original, a priori a.e.-defined $B(f)(\cdot)$, for a.e.\ $\lambda\in\Lambda$,
\item[(iii)] the normalised value $B(f)(\lambda)$ is independent of the choice of $D_0$, up to a
further null set of $\lambda$.
\end{enumerate}
As this is a consequence of the bound $\|B(f)(\lambda)\|_{\op}\le C\|f\|_1$ together with
density of $D_0$, we do not elaborate it here. From now on $B(f)(\lambda)$ denotes this
normalised value. In particular, the Dirac delta property does not depend on the choice of $D_0$,
and is a hypothesis on the normalised $B$.
\end{remark}

\begin{remark}[The differentiation and shift properties at a fixed $\lambda$]\label{rem:fixed-lambda}
The differentiation and shift properties below are stated for a.e.\ $\lambda$ and each $f$. We use the normalisation of Remark~\ref{rem:reps}, together with continuity of $B$ and of the
relevant group actions. A density argument then forces each property to hold for
\emph{every} $f$ simultaneously (and, for the shift property, every $s$), at every $\lambda$ in a
further conull subset of $\Lambda$. We discard the exceptional null set from $\Lambda$ and still denote the resulting set by
$\Lambda$. From now on the differentiation property (\ref{diffprop}) and the shift property (\ref{shiftprop}) hold at every
$\lambda\in\Lambda$, for all $f$ and all $s$.
\end{remark}

First we obtain that any map satisfying Definition~\ref{def:clt} is an integral transform, with the kernel satisfying certain regularity conditions. 

\begin{lemma}[Kernel representation from continuity]
\label{lemE}
Let $B$ be as in Definition~\ref{def:clt}. Then for every $\lambda\in\Lambda$ there is a bounded map $k_\lambda:\Hn\to\mathcal{B}(L^2(\mathbb{R}^n))$ such that
\[
(B(f)(\lambda)\varphi,\psi)=\int_{\Hn}f(z,t)\,(k_\lambda(z,t)\varphi,\psi)\,dz\,dt
\]
for every $f\in L^1(\Hn)$ and $\varphi,\psi\in L^2(\mathbb{R}^n)$.
\end{lemma}

\begin{proof}
Fix $\lambda\in\Lambda$. For $\varphi,\psi\in L^2(\mathbb{R}^n)$ the map $f\mapsto(B(f)(\lambda)\varphi,\psi)$ is linear and satisfies
\begin{align*}
|(B(f)(\lambda)\varphi,\psi)| &\le\|B(f)(\lambda)\|_{\op}\|\varphi\|_2\|\psi\|_2\\
&\le C\|f\|_1\|\varphi\|_2\|\psi\|_2,
\end{align*}
where $C$ is the constant of Definition~\ref{def:clt}, so it defines a continuous linear functional
on $L^1(\Hn)$.
By $L^1$--$L^\infty$ duality on $\Hn$, there is $F_\lambda(\cdot;\varphi,\psi)\in L^\infty(\Hn)$ such that
$$
 (B(f)(\lambda)\varphi,\psi)=\int_{\Hn} f(z,t)\,F_\lambda((z,t);\varphi,\psi)\,dz\,dt,
$$
and
$$\|F_\lambda(\cdot;\varphi,\psi)\|_\infty\le C\|\varphi\|_2\|\psi\|_2.$$
This representative is fixed only up to a null set depending on $\varphi$ and $\psi$, and the
pair $\varphi,\psi$ ranges over an uncountable set, so we now assemble a single kernel,
independent of $\varphi$ and $\psi$, following the method of \cite{LT2012}.

Let $\{\Phi_\alpha\}_{\alpha\in\mathbb{N}^n}$ be the family of Hermite functions, which form an orthonormal basis for $L^2(\mathbb{R}^n)$.  For each pair $\alpha,\beta\in \mathbb{N}^n,$ fix a measurable representative $F_{\alpha,\beta}$ of $F_\lambda(\cdot;\Phi_\alpha,\Phi_\beta)$.

Set
\[
\mathcal{E} := \operatorname{span}_{\mathbb{Q}[i]}\{\Phi_\alpha : \alpha\in\mathbb{N}^n\},
\qquad \overline{\mathcal{E}} = L^2(\mathbb{R}^n).
\]
For $\varphi=\sum a_\alpha\Phi_\alpha,\ \psi=\sum b_\beta\Phi_\beta\in\mathcal{E}$ (finite sums),
sesquilinearity of $F_\lambda(\cdot;\cdot,\cdot)$ gives, for a.e.\ $(z,t)$,
\[
F_\lambda((z,t);\varphi,\psi)=\sum_{\alpha,\beta} a_\alpha\bar{b}_\beta F_{\alpha,\beta}(z,t),
\]
\[
|F_\lambda((z,t);\varphi,\psi)|\le C\|\varphi\|_2\|\psi\|_2.
\]
Let $N\subset\Hn$ be the union, over the countable set $\mathcal{E}\times\mathcal{E}$, of the
exceptional null sets in these two identities. Then $N$ is a null set. Set $E:=\Hn\setminus N$. On $E$, define
\[
S_{(z,t)}(\varphi,\psi):=\sum_{\alpha,\beta} a_\alpha\bar{b}_\beta F_{\alpha,\beta}(z,t),
\qquad \varphi,\psi\in\mathcal{E}.
\]
Then, for every $(z,t)\in E$ and every $\varphi,\psi\in\mathcal{E}$,
\begin{equation}\label{eq:S-agrees}
\left\{
\begin{array}{l}
S_{(z,t)}(\varphi,\psi) = F_\lambda((z,t);\varphi,\psi), \\[4pt]
|S_{(z,t)}(\varphi,\psi)| \le C\|\varphi\|_2\|\psi\|_2.
\end{array}
\right.
\end{equation}
Fix $(z,t)\in E$. By \eqref{eq:S-agrees}, $S_{(z,t)}$ is a sesquilinear form on $\mathcal{E}$ with
\[
\|S_{(z,t)}\| := \sup_{\substack{\varphi,\psi\in\mathcal{E}\\ \|\varphi\|_2,\|\psi\|_2\le 1}} |S_{(z,t)}(\varphi,\psi)| \le C.
\]
Since $\mathcal{E}$ is dense in $L^2(\mathbb{R}^n)$, the map $S_{(z,t)}$ extends uniquely to a bounded
sesquilinear form on $L^2(\mathbb{R}^n)\times L^2(\mathbb{R}^n)$ of the same norm. By the Riesz
representation theorem there is $k_\lambda(z,t)\in\mathcal{B}(L^2(\mathbb{R}^n))$ with
\[
\|k_\lambda(z,t)\|_{\op}\le C, \qquad
(k_\lambda(z,t)\varphi,\psi)=S_{(z,t)}(\varphi,\psi), \quad \textrm{for \ all \ } \varphi,\psi\in L^2(\mathbb{R}^n).
\]
For $(z,t)\notin E$, set $k_\lambda(z,t):=0$.

\noindent \textit{Measurability.} Fix $\varphi,\psi\in L^2(\mathbb{R}^n)$ and $\varphi_m,\psi_m\in\mathcal{E}$
with $\varphi_m\to\varphi$, $\psi_m\to\psi$ in $L^2(\mathbb{R}^n)$. For each $m$, the map
\[
(z,t) \longmapsto (k_\lambda(z,t)\varphi_m,\psi_m) = \sum_{\alpha,\beta} a_\alpha^{(m)}\bar{b}_\beta^{(m)} F_{\alpha,\beta}(z,t), \qquad (z,t)\in E,
\]
is a finite linear combination of the measurable functions $F_{\alpha,\beta}$, hence measurable.
Moreover,
\begin{align*}
\bigl|(k_\lambda(z,t)\varphi_m,\psi_m)-(k_\lambda(z,t)\varphi,\psi)\bigr|
&\le \|k_\lambda(z,t)\|_{\op}\bigl(\|\varphi_m-\varphi\|_2\|\psi_m\|_2+\|\varphi\|_2\|\psi_m-\psi\|_2\bigr) \\
&\le C\bigl(\|\varphi_m-\varphi\|_2\|\psi_m\|_2+\|\varphi\|_2\|\psi_m-\psi\|_2\bigr),
\end{align*}
uniformly in $(z,t)\in E$. So $(k_\lambda(\cdot)\varphi_m,\psi_m) \to (k_\lambda(\cdot)\varphi,\psi)$
uniformly on $E$ as $m\to\infty$. Since uniform limit of measurable functions is measurable, we get that
$(k_\lambda(\cdot)\varphi,\psi)$ is measurable on $\Hn$, for every $\varphi,\psi\in L^2(\mathbb{R}^n)$.
Thus $k_\lambda$ is weakly measurable and bounded on $\Hn$.

\noindent \textit{Kernel representation.} Fix $f\in L^1(\Hn)$. The two sesquilinear forms on
$L^2(\mathbb{R}^n)\times L^2(\mathbb{R}^n)$,
\[
(\varphi,\psi) \longmapsto (B(f)(\lambda)\varphi,\psi), \qquad
(\varphi,\psi) \longmapsto \int_{\Hn} f(z,t)\,(k_\lambda(z,t)\varphi,\psi)\,dz\,dt,
\]
are bounded, and agree on $\mathcal{E}\times\mathcal{E}$ by construction. Since $\mathcal{E}$ is
dense in $L^2(\mathbb{R}^n)$, they agree on $L^2(\mathbb{R}^n)\times L^2(\mathbb{R}^n)$. This proves
the stated representation for every $\varphi,\psi\in L^2(\mathbb{R}^n)$.
\end{proof}

\section{The differentiation-property characterisation}
\label{sec:diff}

In this section, we prove the differentiation-based characterisation of the Fourier transform on the Heisenberg group.

\begin{definition}
\label{def:hdp}
Let $B$ be as in Definition~\ref{def:clt}. We say $B$ satisfies the \emph{Heisenberg differentiation property} if, for every $j=1,\dots,n$ and every $f\in\mathcal{S}(\Hn)$, the following identities hold for a.e.\ $\lambda\in\mathbb{R}^*$:
\begin{equation}
\label{diffprop}
B(iX_jf)(\lambda)=\lambda\,B(f)(\lambda)\,Q_j, \qquad
B(-Y_jf)(\lambda)=B(f)(\lambda)\,P_j .
\end{equation}
Here both sides of each identity are understood as operators applied to every $\varphi\in\mathcal{S}(\mathbb{R}^n)$. By Proposition~\ref{prop:lt23}, the group Fourier transform itself satisfies this property.
\end{definition}

\begin{theorem}
\label{thm:main-diff}
Let $B:L^1(\Hn)\to L^\infty(\mathbb{R}^*,\mathcal{B}(L^2(\mathbb{R}^n)),d\mu)$ be a continuous linear map satisfying the Heisenberg differentiation property for every $j=1,\dots,n$, and the Heisenberg Dirac delta property. Then
\[
B(f)(\lambda)=\pi_\lambda(f)\qquad\text{for every }f\in L^1(\Hn)\text{ and a.e.\ }\lambda\in\mathbb{R}^*.
\]
\end{theorem}

The proof of Theorem~\ref{thm:main-diff} is given at the end of this section.\\

First, we analyse the distributional form of the differentiation property.

Fix $\lambda\in\Lambda$ throughout this section, and write $k_\lambda$ for the kernel of $B$ furnished by Lemma~\ref{lemE}.

\begin{lemma}
\label{lemdist}
Let $B$ satisfy the Heisenberg differentiation property. Then, for every $j=1,\dots,n$, $\varphi\in\mathcal{S}(\mathbb{R}^n)$ and $\psi\in L^2(\mathbb{R}^n)$,
\begin{equation}
\label{distid}
X_j\big[k_\lambda(\cdot;\varphi,\psi)\big]=i\lambda\,k_\lambda(\cdot;Q_j\varphi,\psi),\qquad
Y_j\big[k_\lambda(\cdot;\varphi,\psi)\big]=k_\lambda(\cdot;P_j\varphi,\psi),
\end{equation}
where $k_\lambda(\cdot;\varphi,\psi):=(k_\lambda(\cdot)\varphi,\psi)$ and the derivatives are understood in the sense of \eqref{ibp-ext}.
\end{lemma}

\begin{proof}
Here \eqref{diffprop} is a hypothesis on $B$ only, and no differentiability of $k_\lambda$ is assumed, and the derivatives in \eqref{distid} are distributional. Since $\lambda\in\Lambda$, \eqref{diffprop}
holds at $\lambda$ for every $f\in\mathcal{S}(\Hn)$ (Remark~\ref{rem:fixed-lambda}). Applying \eqref{diffprop} to $f=\chi\in C_c^\infty(\Hn)\subset\mathcal{S}(\Hn)$ and using the
kernel representation of Lemma~\ref{lemE}, the identity 
\[
(B(iX_j\chi)(\lambda)\varphi,\psi) =
\lambda(B(\chi)(\lambda)Q_j\varphi,\psi)
\]
becomes
\[
i\int_{\Hn}(X_j\chi)(g)\,k_\lambda(g;\varphi,\psi)\,dg
= \lambda\int_{\Hn}\chi(g)\,k_\lambda(g;Q_j\varphi,\psi)\,dg
\]
for every $\chi\in C_c^\infty(\Hn)$. By \eqref{ibp-ext}, the left side equals $-i\int_{\Hn}\chi(g)\,X_j[k_\lambda(\cdot;\varphi,\psi)](g)\,dg$.
This gives the first identity in \eqref{distid}. The second identity follows in the same way from the second
identity in \eqref{diffprop}.
\end{proof}

\subsection{Smoothing the kernel}

Fix $\eta\in C_c^\infty(\Hn)$ satisfying
\[
\eta\ge0,\qquad \int_{\Hn}\eta=1,
\]
and set 
\[
\eta_\varepsilon(z,t):=\varepsilon^{-(2n+2)}\eta(z/\varepsilon,t/\varepsilon^2), \textrm{ \ for \ } \varepsilon>0.
\]
The exponent $2n+2$ is the homogeneous dimension of $\Hn$ under the dilations $(z,t)\mapsto(rz,r^2t)$, so $\int_{\Hn}\eta_\varepsilon=1$ for every $\varepsilon$, and $(\eta_\varepsilon)_{\varepsilon>0}$ is an approximate identity supported in shrinking neighbourhoods of $(0,0)$.

For $\varphi,\psi\in L^2(\mathbb{R}^n)$, define
\begin{equation}
\label{molldef}
k_\lambda^\varepsilon(g;\varphi,\psi):=\int_{\Hn}\eta_\varepsilon(h)\,k_\lambda(h^{-1}g;\varphi,\psi)\,dh,\qquad g\in\Hn.
\end{equation}

The following result illustrates that smoothing preserves the differentiation identities.

\begin{lemma}
\label{lem:smooth}
For each $\varepsilon>0$, $g\mapsto k_\lambda^\varepsilon(g;\varphi,\psi)$ is $C^\infty$ on $\Hn$, and, for $\varphi\in\mathcal{S}(\mathbb{R}^n)$,
\begin{equation}
\label{molleq}
X_j\big[k_\lambda^\varepsilon(\cdot;\varphi,\psi)\big]=i\lambda\,k_\lambda^\varepsilon(\cdot;Q_j\varphi,\psi),\qquad
Y_j\big[k_\lambda^\varepsilon(\cdot;\varphi,\psi)\big]=k_\lambda^\varepsilon(\cdot;P_j\varphi,\psi)
\end{equation}
classically, for every $g\in\Hn$.
\end{lemma}

\begin{proof}
Since Haar measure is bi-invariant, substituting $u=h^{-1}g,$ we get
\[
k_\lambda^\varepsilon(g;\varphi,\psi)=\int_{\Hn}\eta_\varepsilon(gu^{-1})\,k_\lambda(u;\varphi,\psi)\,du.
\]
Since $\eta_\varepsilon\in C_c^\infty(\Hn)$ and $k_\lambda(\cdot;\varphi,\psi)\in L^\infty(\Hn)$, the $g$-derivatives of $\eta_\varepsilon(gu^{-1})$ are smooth and compactly supported in $u$, locally uniformly in $g$. Differentiation under the integral sign is therefore justified by dominated convergence, giving $k_\lambda^\varepsilon(\cdot;\varphi,\psi)\in C^\infty(\Hn)$.

Fix $\chi\in C_c^\infty(\Hn)$ and write $\chi_h(g'):=\chi(hg')$. Applying \eqref{leftcomm} with $h$
replaced by $h^{-1}$, so that $\chi_h=L_{h^{-1}}\chi$, we get
\[
(X_j\chi)(hg') = \big(X_j(\chi_h)\big)(g').
\]
Since $\chi_h\in C_c^\infty(\Hn)$ and $k_\lambda(\cdot;\varphi,\psi)\in L^\infty(\Hn)$, equation \eqref{ibp-ext}
gives
\[
\int_{\Hn}k_\lambda(g';\varphi,\psi)\,(X_j\chi_h)(g')\,dg'
= -\int_{\Hn}\big(X_jk_\lambda(\cdot;\varphi,\psi)\big)(g')\,\chi_h(g')\,dg'.
\]
Substituting $g'=h^{-1}g$, we get
\begin{align*}
\int_{\Hn}k_\lambda^\varepsilon(g;\varphi,\psi)(X_j\chi)(g)\,dg
&=\int_{\Hn}\eta_\varepsilon(h)\left[\int_{\Hn}k_\lambda(g';\varphi,\psi)(X_j\chi_h)(g')\,dg'\right]dh\\
&=-\int_{\Hn}\eta_\varepsilon(h)\left[\int_{\Hn}\big(X_j k_\lambda(\cdot;\varphi,\psi)\big)(g')\,\chi_h(g')\,dg'\right]dh\\
&=-i\lambda\int_{\Hn}\eta_\varepsilon(h)\left[\int_{\Hn}k_\lambda(g';Q_j\varphi,\psi)\,\chi(hg')\,dg'\right]dh\\
&=-i\lambda\int_{\Hn}k_\lambda^\varepsilon(g;Q_j\varphi,\psi)\,\chi(g)\,dg,
\end{align*}
using \eqref{distid} in the third line and $g=hg'$ in the last. Thus, as distributions,
\[
X_j\big[k_\lambda^\varepsilon(\cdot;\varphi,\psi)\big]=i\lambda\,k_\lambda^\varepsilon(\cdot;Q_j\varphi,\psi).
\]
By \eqref{ibp}, the left-hand side equals the classical derivative of the $C^\infty$ function $k_\lambda^\varepsilon(\cdot;\varphi,\psi)$. By \eqref{molldef}, the right-hand side is continuous, hence smooth as above. Two continuous functions agreeing as distributions agree pointwise, so the identity holds in the classical sense, at every $g$. The computation corresponding to $Y_j$ follows identically, with no factor of $i\lambda$, using the second identity in \eqref{distid}.
\end{proof}

Since $k_\lambda(\cdot;\varphi,\psi)$ is sesquilinear in $(\varphi,\psi)$ and $\|k_\lambda\|_\infty<\infty$, so is $(\varphi,\psi)\mapsto k_\lambda^\varepsilon(g;\varphi,\psi)$, as a bounded form on $L^2(\mathbb{R}^n)\times L^2(\mathbb{R}^n)$ of norm $\le\|k_\lambda\|_\infty$.

By Lemma~\ref{lem:smooth}, for $\varphi\in\mathcal{S}(\mathbb{R}^n)$ and $\psi\in L^2(\mathbb{R}^n)$,
\begin{equation}
\label{Kepssys}
X_j(K_\varepsilon(g)\varphi,\psi)=i\lambda\,(K_\varepsilon(g)Q_j\varphi,\psi),\qquad Y_j(K_\varepsilon(g)\varphi,\psi)=(K_\varepsilon(g)P_j\varphi,\psi),\qquad g\in\Hn.
\end{equation}

\subsection{Solving the smooth system}

\begin{lemma}
\label{lemB}
Let $R:\Hn\to\mathcal{B}(L^2(\mathbb{R}^n))$ be a bounded map such that, for all $\varphi\in\mathcal{S}(\mathbb{R}^n)$ and $\psi\in L^2(\mathbb{R}^n)$, the function $g\mapsto(R(g)\varphi,\psi)$ is $C^\infty$ and, for every $j=1,\dots,n$,
\begin{equation}
\label{Rsys}
X_j(R\varphi,\psi)=i\lambda\,(RQ_j\varphi,\psi),\qquad Y_j(R\varphi,\psi)=(RP_j\varphi,\psi).
\end{equation}
Then $R(z,t)=R(0,0)\,\pi_\lambda(z,t)$.
\end{lemma}

\begin{proof}
First we recover the dependence of $R$ on the central variable $t$ from the commutator relation $[X_j,Y_j]=-\partial_t$. \\

Fix $\varphi \in\mathcal{S}(\mathbb{R}^n)$ and $\psi\in L^2(\mathbb{R}^n)$. Applying \eqref{Rsys} twice, with $P_j\varphi,Q_j\varphi\in\mathcal{S}(\mathbb{R}^n)$, we get
\[
X_jY_j(R\varphi,\psi)=i\lambda(RQ_jP_j\varphi,\psi),\qquad
Y_jX_j(R\varphi,\psi)=i\lambda(RP_jQ_j\varphi,\psi).
\]
By \eqref{CCR},
\[
[X_j,Y_j](R\varphi,\psi)=i\lambda(R[Q_j,P_j]\varphi,\psi)=-i\lambda(R\varphi,\psi).
\]
Since $[X_j,Y_j]=-\partial_t$ by \eqref{bracket}, this yields
\[
\partial_t(R\varphi,\psi)=i\lambda(R\varphi,\psi),
\]
hence
\[
(R(z,t)\varphi,\psi)=e^{i\lambda t}(R(z,0)\varphi,\psi).
\]
As $\mathcal{S}(\mathbb{R}^n)$ is dense and $R$ is bounded,
\[
R(z,t)=e^{i\lambda t}\rho(z) \ \text{with \ } \rho(z):=R(z,0).
\]
Substituting into \eqref{Rsys}, for $\varphi\in\mathcal{S}(\mathbb{R}^n)$,
\begin{equation}
\label{rhosys}
\partial_{x_j}(\rho\varphi,\psi)=i\lambda\Big(\rho\big(Q_j-\tfrac{y_j}{2}\big)\varphi,\psi\Big),\qquad
\partial_{y_j}(\rho\varphi,\psi)=\Big(\rho\big(P_j+\tfrac{i\lambda x_j}{2}\big)\varphi,\psi\Big).
\end{equation}

Now we obtain the corresponding derivatives of the known inverse $\pi_\lambda(z)^{-1}$. \\
Fix $\varphi\in\mathcal{S}(\mathbb{R}^n)$ and put $\eta(z):=\pi_\lambda(z)^{-1}\varphi=\pi_\lambda(-z,0)\varphi$. Explicitly,
\[
\eta(z)(\xi)=e^{i\lambda(-x\cdot\xi+\frac12x\cdot y)}\varphi(\xi-y).
\]
As in the proof of Lemma~\ref{lemA}, this is a smooth $\mathcal{S}(\mathbb{R}^n)$-valued function of $z$, differentiable in the norm of $L^2(\mathbb{R}^n)$. Differentiating this formula with respect to $x_j$ and$y_j$ give
\begin{equation}
\label{etasys}
\partial_{x_j}\eta=-i\lambda\Big(Q_j-\tfrac{y_j}{2}\Big)\eta,\qquad \partial_{y_j}\eta=-\Big(P_j+\tfrac{i\lambda x_j}{2}\Big)\eta.
\end{equation}
For obtaining the second identity we use that 
\[
\partial_{y_j}\big[e^{i\lambda(-x\cdot\xi+\frac12x\cdot y)}\big]=\tfrac{i\lambda x_j}{2}\eta(\xi),
\]
\[
\partial_{y_j}\big[\varphi(\xi-y)\big]=-e^{i\lambda(-x\cdot\xi+\frac12x\cdot y)}(\partial_j\varphi)(\xi-y),
\]
 and 
\[
 P_j\eta(\xi)=-i\lambda x_j\eta(\xi)+e^{i\lambda(-x\cdot\xi+\frac12x\cdot y)}(\partial_j\varphi)(\xi-y).
 \]

Finally we compare the two forms obtained above, to show that $\rho(z)\pi_\lambda(z)^{-1}$ is a constant, independent of $z$.

For $\psi\in L^2(\mathbb{R}^n)$ put $u(z):=(\rho(z)\eta(z),\psi)$. Fix $z_0\in\mathbb{C}^n$ and let $z=z_0+re_j$ with $r\in\mathbb{R}$, $r\to0$. Then
\[
u(z)-u(z_0) = \big[(\rho(z)\eta(z_0),\psi)-(\rho(z_0)\eta(z_0),\psi)\big] + \big(\rho(z)(\eta(z)-\eta(z_0)),\psi\big).
\]
Since $\eta(z_0)\in\mathcal{S}(\mathbb{R}^n)$, the function $z\mapsto(\rho(z)\eta(z_0),\psi)$ is smooth, and by \eqref{rhosys} its derivative with respect to $x_j$ at $z_0$ is
\[
i\lambda\big(\rho(z_0)(Q_j-\tfrac{y_j}{2})\eta(z_0),\psi\big).
\]
For the second term, $\tfrac1r(\eta(z)-\eta(z_0))\to\partial_{x_j}\eta(z_0)$ in $L^2(\mathbb{R}^n)$, because $\eta$ is norm-differentiable. Since $\rho$ is bounded, it follows that
\[
\Big(\rho(z)\tfrac1r\big(\eta(z)-\eta(z_0)\big),\psi\Big)-\big(\rho(z)\partial_{x_j}\eta(z_0),\psi\big)\to0.
\]
Since $\partial_{x_j}\eta(z_0)\in\mathcal{S}(\mathbb{R}^n)$ and $z\mapsto(\rho(z)\phi,\psi)$ is continuous for every $\phi\in\mathcal{S}(\mathbb{R}^n),$ we get
\[
\big(\rho(z)\partial_{x_j}\eta(z_0),\psi\big)\to\big(\rho(z_0)\partial_{x_j}\eta(z_0),\psi\big) \quad \text{as } z\to z_0.
\]
Thus
\[
\partial_{x_j}u(z_0)=i\lambda\Big(\rho(z_0)\big(Q_j-\tfrac{y_j}{2}\big)\eta(z_0),\psi\Big)+\big(\rho(z_0)\partial_{x_j}\eta(z_0),\psi\big).
\]
The same argument with $z=z_0+ire_j$ gives the corresponding formula for $\partial_{y_j}u(z_0)$. 
Combining \eqref{rhosys} and \eqref{etasys} we obtain
\[
\partial_{x_j}u=i\lambda\Big(\rho\big(Q_j-\tfrac{y_j}{2}\big)\eta,\psi\Big)+\Big(\rho\big(-i\lambda(Q_j-\tfrac{y_j}{2})\eta\big),\psi\Big)=0.
\]
Similarly, we get $\partial_{y_j}u=0$. As $\mathbb{C}^n$ is connected,
\[
u(z)=u(0)=(\rho(0)\varphi,\psi),
\]
that is,
\[
(\rho(z)\pi_\lambda(z)^{-1}\varphi,\psi)=(R(0,0)\varphi,\psi) \quad \text{for all } \varphi\in\mathcal{S}(\mathbb{R}^n) \text{ and } \  \psi\in L^2(\mathbb{R}^n),
\]
with $\rho(0)=R(0,0)$.\\
Since $\mathcal{S}(\mathbb{R}^n)$ is dense in $L^2(\mathbb{R}^n)$ and both sides are bounded,
\[
\rho(z)\pi_\lambda(z)^{-1}=R(0,0), \quad \text{i.e.}\quad \rho(z)=R(0,0)\,\pi_\lambda(z).
\]
Thus we get $R(z,t)=e^{i\lambda t}\rho(z)=R(0,0)\,\pi_\lambda(z,t)$, as claimed.
\end{proof}

\begin{corollary}
\label{cor:Keps}
For every $\varepsilon>0$, $K_\varepsilon(z,t)=K_\varepsilon(0,0) \,\pi_\lambda(z,t)$ for every $(z,t)\in\Hn$.
\end{corollary}

\begin{proof}
By Lemma~\ref{lem:smooth} and \eqref{Kepssys}, the map $R:=K_\varepsilon$ is bounded, weakly smooth, and satisfies \eqref{Rsys}. Lemma~\ref{lemB} therefore gives $K_\varepsilon(z,t)=K_\varepsilon(0,0)\,\pi_\lambda(z,t)=D_\varepsilon\,\pi_\lambda(z,t)$.
\end{proof}

\subsection{Removing the smoothing}

\begin{lemma}[Passage to the limit]
\label{lem:limit}
There is $D(\lambda)\in\mathcal{B}(L^2(\mathbb{R}^n))$, with $\|D(\lambda)\|_{\op}\le\|k_\lambda\|_\infty$, such that
\[
k_\lambda(z,t)=D(\lambda)\,\pi_\lambda(z,t)\qquad\text{for a.e.\ }(z,t)\in\Hn.
\]
\end{lemma}

\begin{proof}
For each fixed pair $\varphi,\psi\in L^2(\mathbb{R}^n)$, the function $u:=k_\lambda(\cdot;\varphi,\psi)$ lies in $L^\infty(\Hn)\subset L^1_{\mathrm{loc}}(\Hn)$, and $k_\lambda^\varepsilon(\cdot;\varphi,\psi)=\eta_\varepsilon*u$. The functions $(\eta_\varepsilon)$ are supported in shrinking neighbourhoods of the identity and have integral one. Since translation is continuous in $L^1_{\mathrm{loc}}$, we get
\[
(\eta_\varepsilon*u)(g)-u(g)=\int_{\Hn}\eta_\varepsilon(h)\big[u(h^{-1}g)-u(g)\big]\,dh \longrightarrow 0 \quad \text{in } L^1_{\mathrm{loc}}(\Hn) \text{ as } \varepsilon\to0,
\]
i.e.\ $\eta_\varepsilon*u\to u$ in $L^1_{\mathrm{loc}}(\Hn)$.

Applying this to the countably many pairs $(\Phi_\alpha,\Phi_\beta)$ and using a diagonal argument, there exist a sequence $\varepsilon_m\to0$ and a set $E\subset\Hn$ whose complement is a null set such that
\[
k_\lambda^{\varepsilon_m}(g;\Phi_\alpha,\Phi_\beta)\to k_\lambda(g;\Phi_\alpha,\Phi_\beta) \qquad \text{for every } \alpha,\beta \text{ and every } g\in E.
\]
Since $\|K_{\varepsilon}(g)\|_{\op}\le\|k_\lambda\|_\infty$ and the set of Hermite functions is total,
\[
K_{\varepsilon_m}(g)\to k_\lambda(g) \quad \text{weakly, for every } g\in E.
\]

For $g_0\in E,$ we know that $\pi_\lambda(g_0)^{-1}=\pi_\lambda(g_0)^*=\pi_\lambda(g_0^{-1})$ is a bounded operator, since $\pi_\lambda$ is a unitary representation. By Corollary~\ref{cor:Keps},
\[
K_{\varepsilon_m}(0,0)=K_{\varepsilon_m}(g_0)\pi_\lambda(g_0)^{-1}\to k_\lambda(g_0)\pi_\lambda(g_0)^{-1}=:D(\lambda) \quad \text{weakly},
\]
since right multiplication by a fixed bounded operator preserves weak convergence. Moreover,
\[
\|D(\lambda)\|_{\op}\le\|k_\lambda\|_\infty, \qquad \text{since } \|K_{\varepsilon}(0,0)\|_{\op}=\|K_{\varepsilon}(g_0)\|_{\op}\le\|k_\lambda\|_\infty.
\]
Another application of Corollary~\ref{cor:Keps} gives
\[
K_{\varepsilon_m}(g)=D_{\varepsilon_m}\pi_\lambda(g)\to D(\lambda)\pi_\lambda(g) \quad \text{weakly, for every } g\in\Hn.
\]
Comparing the two limits of $K_{\varepsilon_m}(g)$ at each $g\in E$ gives $k_\lambda(g)=D(\lambda)\pi_\lambda(g)$.
\end{proof}

The following result gives a precise description of the map $D(\lambda)$.

\begin{lemma}
\label{lemCD}
Let $D(\lambda)$ be as in \Cref{lem:limit}. Then 
\begin{enumerate}
\item[\textup{(i)}] $B(f)(\lambda) = D(\lambda)\,\pi_\lambda(f)$ for every $f\in L^1(\Hn)$.
\item[\textup{(ii)}] If, in addition, $B$ satisfies the Heisenberg Dirac delta property, then $D(\lambda)=I$ for a.e.\ $\lambda$.
\end{enumerate}
\end{lemma}

\begin{proof}
(i) By Lemma~\ref{lem:limit} and Lemma~\ref{lemE},
\[
B(f)(\lambda)=\int f(g)k_\lambda(g)\,dg=D(\lambda)\int f(g)\pi_\lambda(g)\,dg=D(\lambda)\pi_\lambda(f).
\]

(ii) Let $(\delta_U)$ be the Dirac delta family in Definition~\ref{def:ddp}. By (i), $B(\delta_U)(\lambda)=D(\lambda)\pi_\lambda(\delta_U)$. Since $\pi_\lambda$ is strongly continuous and $\pi_\lambda(0,0)=I$, $\pi_\lambda(\delta_U)\to I$ strongly as $U\to(0,0)$, so $B(\delta_U)(\lambda)\to D(\lambda)$. The Dirac delta property forces this limit to be $I$ for a.e.\ $\lambda$.
\end{proof}

Now we are ready to assemble the proof of the characterisation theorem in terms of differentiation property.

\begin{proof}[Proof of Theorem~\ref{thm:main-diff}]
Combine parts (i) and (ii) of Lemma~\ref{lemCD}. This gives $B(f)(\lambda)=\pi_\lambda(f)$ for every $f\in L^1(\Hn)$ and a.e.\ $\lambda$.
\end{proof}

\begin{remark}
If $k_\lambda$ happens already to be weakly smooth, with the differentiation identities holding classically, Lemma~\ref{lemB} applies directly to $R:=k_\lambda$. In this case, there is no need for smoothing or limiting step, and we get $k_\lambda(z,t)=k_\lambda(0,0)\,\pi_\lambda(z,t)$. This occurs, for instance, if $B$ is given from the outset by such a kernel, as in Theorem~\ref{thm:kruk}. This is the special case in which the argument reduces to Lemma~\ref{lemB} alone.
\end{remark}

\section{The shift-property characterisation}
\label{sec:shift}
In this section, we obtain the shift-based characterisation of the group Fourier transform on the Heisenberg group.\\

First we setup the notation and terminology needed for stating the main result.

\begin{definition}
\label{def:hsp}
Let $B$ be as in Definition~\ref{def:clt}. We say $B$ satisfies the \emph{Heisenberg shift property} if, for every $j=1,\dots,n$, $s\in\mathbb{R}$, and $f\in L^1(\Hn)$,
\begin{equation}
\label{shiftprop}
B(R_{(se_j,0)}f)(\lambda)=B(f)(\lambda)\,e^{-is\lambda Q_j},\qquad
B(R_{(ise_j,0)}f)(\lambda)=B(f)(\lambda)\,e^{-sP_j}
\end{equation}
for a.e.\ $\lambda\in\mathbb{R}^*$.
\end{definition}

By \eqref{pivals}, the shift property \eqref{shiftprop} is precisely Proposition~\ref{prop:lt-prop3}'s covariance identity $\pi_\lambda(R_hf)=\pi_\lambda(f)\pi_\lambda(h)^*$, applied at $h=(se_j,0)$ and $h=(ise_j,0)$. In particular, the group Fourier transform itself satisfies the Heisenberg shift property.

The following is the shift-based characterisation for the group Fourier transform.

\begin{theorem}
\label{thm:main-shift}
Let $B:L^1(\Hn)\to L^\infty(\mathbb{R}^*,\mathcal{B}(L^2(\mathbb{R}^n)),d\mu)$ be a continuous linear map satisfying the Heisenberg shift property and the Heisenberg Dirac delta property. Then
\[
B(f)(\lambda)=\pi_\lambda(f)\qquad\text{for every }f\in L^1(\Hn)\text{ and a.e.\ }\lambda\in\mathbb{R}^*.
\]
\end{theorem}

Throughout this section, $\lambda\in\Lambda$ is fixed and $k_\lambda$ denotes the kernel of $B$ furnished by Lemma~\ref{lemE}.

We obtain the following covariance formula for the kernel.
\begin{lemma}
\label{lemF}
For every $h\in\Hn$,
\begin{equation}
\label{genformula}
B(R_hf)(\lambda)=\int_{\Hn}f(g)\,k_\lambda(gh^{-1})\,dg,\qquad g:=(z,t).
\end{equation}
Consequently, the Heisenberg shift property at $\lambda$ (that is, \eqref{shiftprop} at $\lambda$ for all $f$ and $s$) is equivalent to the a.e.\ identity
\begin{equation}
\label{covreln}
k_\lambda(gk)=k_\lambda(g)\,\pi_\lambda(k)\qquad\text{for a.e.\ }g\in\Hn,
\end{equation}
holding for every $k$ in the family $\mathcal{K}:=\{(se_j,0),(ise_j,0):s\in\mathbb{R},\,j=1,\dots,n\}$.
\end{lemma}

\begin{proof}
By Lemma~\ref{lemE},
\[
B(R_hf)(\lambda) = \int_{\Hn} f(gh)\,k_\lambda(g)\,dg.
\]
Substituting $g'=gh$ and using right-invariance of Haar measure on the unimodular group $\Hn$ gives \eqref{genformula}.

Now take $h=k\in\mathcal K$. Since $\lambda\in\Lambda$, the shift property \eqref{shiftprop} holds at $\lambda$ for every $f\in L^1(\Hn)$ (Remark~\ref{rem:fixed-lambda}), i.e.,
\[
B(R_kf)(\lambda)=B(f)(\lambda)\pi_\lambda(k)^*=B(f)(\lambda)\pi_\lambda(k^{-1}).
\]
Comparing with \eqref{genformula}, pairing both sides with $\varphi,\psi\in L^2(\mathbb{R}^n)$, and using Lemma~\ref{lemE},
\[
\int_{\Hn}f(g)\Big[\big(k_\lambda(gk^{-1})\varphi,\psi\big)-\big(k_\lambda(g)\pi_\lambda(k^{-1})\varphi,\psi\big)\Big]\,dg=0 \qquad \text{for every } f\in L^1(\Hn).
\]
The expression in square brackets is a bounded measurable function of $g$, since $\|k_\lambda\|_\infty\le C$, $\pi_\lambda(k^{-1})$ is unitary, and right translation preserves Haar measure. Since a bounded function integrating to zero against every $f\in L^1(\Hn)$ vanishes almost everywhere, for fixed $\varphi,\psi$,
\[
\big(k_\lambda(gk^{-1})\varphi,\psi\big)=\big(k_\lambda(g)\pi_\lambda(k^{-1})\varphi,\psi\big) \qquad \text{for a.e.\ } g\in\Hn.
\]
Let $\varphi,\psi$ range over the countable set of finite $\mathbb{Q}[i]$-linear combinations of Hermite functions used in the proof of Lemma~\ref{lemE}. The union of the corresponding null sets is again null. Outside this null set, the bounded sesquilinear forms 
\[
(\varphi,\psi)\mapsto(k_\lambda(gk^{-1})\varphi,\psi) \textrm{ \ and \ } (\varphi,\psi)\mapsto(k_\lambda(g)\pi_\lambda(k^{-1})\varphi,\psi)
\]
agree on a dense subset of $L^2(\mathbb{R}^n)\times L^2(\mathbb{R}^n)$, and hence they agree everywhere. Thus
\[
k_\lambda(gk^{-1})=k_\lambda(g)\,\pi_\lambda(k^{-1}) \qquad \text{for a.e.\ } g\in\Hn.
\]
Replacing $k$ by $k^{-1}$, which also ranges over $\mathcal K$, gives \eqref{covreln}.

Conversely, if \eqref{covreln} holds for every $k\in\mathcal K$, substituting $k_\lambda(gk^{-1})=k_\lambda(g)\pi_\lambda(k^{-1})$ in \eqref{genformula} with $h=k$ gives
\[
B(R_kf)(\lambda)=B(f)(\lambda)\,\pi_\lambda(k)^*,
\]
which is the shift property.
\end{proof}

\begin{lemma}[Closure and generation]
\label{lemG}
For every $(z,t)\in\Hn$, relation \eqref{covreln} holds, for a.e.\ $g\in\Hn$ (with an exceptional null set that may depend on $(z,t)$), taking $k=(z,t)$.
\end{lemma}

\begin{proof}
Suppose \eqref{covreln} holds for $k_1$ and for $k_2$, each for a.e.\ $g$, say outside null sets $N_1$ and $N_2$ respectively. So $N_2k_1^{-1}$ is also null, and evaluating \eqref{covreln} for $k_2$ at $gk_1$, then \eqref{covreln} for $k_1$ at $g$, gives
\[
k_\lambda(gk_1k_2)=k_\lambda(gk_1)\pi_\lambda(k_2)=k_\lambda(g)\pi_\lambda(k_1)\pi_\lambda(k_2)=k_\lambda(g)\pi_\lambda(k_1k_2)\qquad\text{for }g\notin N_1\cup N_2k_1^{-1}.
\]
Thus \eqref{covreln} holds for $k_1k_2$. Applying \eqref{covreln} at $k$ at the point $gk^{-1}$ gives $k_\lambda(g)=k_\lambda(gk^{-1})\pi_\lambda(k)$, that is,
\[
k_\lambda(gk^{-1})=k_\lambda(g)\,\pi_\lambda(k^{-1}) \qquad \text{for a.e.\ } g,
\]
so \eqref{covreln} also holds for $k^{-1}$. Hence the set of $k\in\Hn$ for which \eqref{covreln} holds (for a.e.\ $g$) is a subgroup of $\Hn$ containing $\mathcal K$, and it suffices to exhibit, for an arbitrary $(z,t)\in\Hn$, a finite product of elements of $\mathcal K$ equal to $(z,t)$.

Write $z=(x_1+iy_1,\dots,x_n+iy_n)$ and set 
\[
\Delta t:=t+\tfrac12\sum_{j=1}^n x_jy_j.
\]
For $g=(z,0)$, $h=(w,0)\in\Hn$, the commutator identity $ghg^{-1}h^{-1}=\big(0,\operatorname{Im}(z\bar w)\big)$ holds. Taking $z=ae_1$, $w=ibe_1$ with $a,b\in\mathbb{R}$, we get $z\bar w=-iab$, and
\[
(ae_1,0)(ibe_1,0)(-ae_1,0)(-ibe_1,0)=(0,-ab),
\]
Taking $a=1$, $b=-\Delta t$ yields the central translation
\[
(e_1,0)(-i\Delta t\,e_1,0)(-e_1,0)(i\Delta t\,e_1,0)=(0,\Delta t).
\]
By the same computation, for each $j$, we get
\[
(x_je_j,0)(iy_je_j,0)=\big(x_je_j+iy_je_j,\,-\tfrac12x_jy_j\big).
\]
Since the Hermitian form $z\cdot\bar w$ vanishes whenever $z,w$ have disjoint coordinate support, elements of $\mathcal{K}$ attached to distinct coordinates $j\ne k$ commute and contribute no cross term to the central coordinate. Hence
\[
(x_1e_1,0)(iy_1e_1,0)\cdots(x_ne_n,0)(iy_ne_n,0)=\Big(z,\ -\tfrac12\sum_{j=1}^n x_jy_j\Big).
\]
Multiplying by the central correction above, we get
\[
(x_1e_1,0)(iy_1e_1,0)\cdots(x_ne_n,0)(iy_ne_n,0)\cdot(e_1,0)(-i\Delta t\,e_1,0)(-e_1,0)(i\Delta t\,e_1,0)=(z,t),
\]
a product of $2n+4$ elements of $\mathcal{K}$. By the closure property established above, \eqref{covreln} holds, for a.e.\ $g$, at $k=(z,t)$.
\end{proof}

The following result illustrates that the kernel collapses to a much simpler form.
\begin{lemma}
\label{lemH}
There is $D(\lambda)\in\mathcal{B}(L^2(\mathbb{R}^n))$ such that $k_\lambda(g)=D(\lambda)\,\pi_\lambda(g)$ for a.e.\ $g\in\Hn$.
\end{lemma}

\begin{proof}
Fix $\alpha,\beta\in\mathbb{N}^n$. By Lemma~\ref{lemG}, for every fixed $(z,t)\in\Hn$ the function
\begin{equation}
\label{eq:diffab}
g\longmapsto\big(k_\lambda(g\cdot(z,t))\Phi_\alpha,\Phi_\beta\big)-\big(k_\lambda(g)\pi_\lambda(z,t)\Phi_\alpha,\Phi_\beta\big)
\end{equation}
vanishes for a.e.\ $g\in\Hn$. We first check that the same expression is a measurable function of $(g,z,t)\in\Hn\times\Hn$.

Consider the first term. By Lemma~\ref{lemE}, $(k_\lambda(\cdot)\Phi_\alpha,\Phi_\beta)$ is measurable, so it agrees with a Borel function, say, $\sigma_{\alpha\beta}:\Hn\to\mathbb{C}$ outside a Borel null set $M\subset\Hn$. The map $\nu(g,z,t):=g\cdot(z,t)$ is continuous, so $\sigma_{\alpha\beta}\circ\nu$ is Borel on $\Hn\times\Hn$. The set of points $(g,z,t)$ at which the first term differs from $\sigma_{\alpha\beta}(g\cdot(z,t))$ is contained in $\nu^{-1}(M)$. This preimage is Borel, and for each fixed $g$ its slice in the variable $(z,t)$ is a left translate of $M$, hence null. By Tonelli's theorem $\nu^{-1}(M)$ is null in $\Hn\times\Hn$, so the first term is measurable.

Consider the second term. For each fixed $g$ the map $v\mapsto(k_\lambda(g)v,\Phi_\beta)$ is a bounded linear functional on $L^2(\mathbb{R}^n)$. Expanding $\pi_\lambda(z,t)\Phi_\alpha$ in the orthonormal basis $\{\Phi_\gamma\}_{\gamma\in\mathbb{N}^n}$ gives
\[
\big(k_\lambda(g)\pi_\lambda(z,t)\Phi_\alpha,\Phi_\beta\big)=\sum_{\gamma\in\mathbb{N}^n}\big(\pi_\lambda(z,t)\Phi_\alpha,\Phi_\gamma\big)\big(k_\lambda(g)\Phi_\gamma,\Phi_\beta\big),
\]
where each term is the product of a continuous function of $(z,t)$ and a measurable function of $g$, and the series converges at every point. Hence the second term is measurable as well.

Since the expression is measurable and, for every fixed $(z,t)$, vanishes for a.e.\ $g$, Tonelli's theorem shows that it vanishes a.e.\ on $\Hn\times\Hn$. Again applying Tonelli's theorem in the other order, for a.e.\ $g\in\Hn$, we get that it vanishes for a.e.\ $(z,t)\in\Hn$.

This holds for every pair $\alpha,\beta$. Taking the countable intersection over all pairs, there is a single set $E\subset\Hn$ whose complement is null such that, for every $g\in E$ and a.e.\ $(z,t)\in\Hn$, the expression \eqref{eq:diffab} vanishes for all $\alpha,\beta$ simultaneously. For such $g,(z,t)$ the bounded operators $k_\lambda(g\cdot(z,t))$ and $k_\lambda(g)\pi_\lambda(z,t)$ agree on the Hermite basis. Since both these operators are bounded and the finite linear combinations of Hermite functions are dense in $L^2(\mathbb{R}^n)$, the two operators coincide. Thus, for every $g\in E$,
\[
k_\lambda(g\cdot(z,t))=k_\lambda(g)\,\pi_\lambda(z,t)\qquad\text{for a.e.\ }(z,t)\in\Hn.
\]

Fix $g_0\in E$. The substitution $g:=g_0\cdot(z,t)$, equivalently $(z,t)=g_0^{-1}g$, is a bijection of $\Hn$. Since $\Hn$ is unimodular, it preserves Haar measure, so it turns ``for a.e.\ $(z,t)$'' into ``for a.e.\ $g$''. Thus we have
\[
k_\lambda(g)=k_\lambda(g_0)\,\pi_\lambda(g_0^{-1}g)=k_\lambda(g_0)\,\pi_\lambda(g_0)^{-1}\pi_\lambda(g)=:D(\lambda)\,\pi_\lambda(g)\qquad\text{for a.e.\ }g\in\Hn.
\]
\end{proof}

We are now ready to prove the shift-based characterisation.

\begin{proof}[Proof of Theorem~\ref{thm:main-shift}]
By Lemma~\ref{lemH} and Lemma~\ref{lemE},
\begin{align*}
B(f)(\lambda)&=\int_{\Hn}f(g)\,k_\lambda(g)\,dg\\
&=D(\lambda)\int_{\Hn}f(g)\,\pi_\lambda(g)\,dg=D(\lambda)\,\pi_\lambda(f)
\end{align*}
for every $f\in L^1(\Hn)$, the middle equality being unaffected by altering $k_\lambda$ on a null set. Let $(\delta_U)$ be the Heisenberg Dirac delta family of Definition~\ref{def:ddp}. Then
\[
B(\delta_U)(\lambda)=D(\lambda)\,\pi_\lambda(\delta_U)\longrightarrow D(\lambda)
\]
strongly as $U\to(0,0)$, using strong continuity of $\pi_\lambda$ at the identity (Proposition~\ref{prop:pi-ddp}). The Dirac delta property forces this limit to equal $I$ for a.e.\ $\lambda$, so $D(\lambda)=I$ there, and $B(f)(\lambda)=\pi_\lambda(f)$.
\end{proof}

\section{Further discussion}
\label{sec:furtherdisc}
In this section, we illustrate the equivalence of the hypotheses of the shift- and differentiation-based characterisations proved above. Then we discuss the differences in the nature of the hypotheses of our results  and those of \cite{Krukowski2}.\\

The differentiation property and the shift property are the Lie-algebra and the Lie-group forms of the same condition. The differentiation property consists of finitely many identities, one for each generator $X_j,Y_j$ of the Lie algebra of $\Hn$, while the shift property consists of an identity for every element of each one-parameter subgroup $s\mapsto(se_j,0)$, $s\mapsto(ise_j,0)$ generated by these fields. Continuity of $B$ allows one to pass from one form to the other, giving the following equivalence. We keep Theorems~\ref{thm:main-diff} and~\ref{thm:main-shift} separate because their proofs are different, even though their hypotheses coincide.

\begin{proposition}[Lie algebra versus Lie group]
\label{prop:equiv}
Let $B$ be as in Definition~\ref{def:clt}. Then $B$ satisfies the Heisenberg differentiation property if and only if $B$ satisfies the Heisenberg shift property.
\end{proposition}

\begin{proof}
Fix $j\in\{1,\dots,n\}$. By Remark~\ref{rem:fixed-lambda}, both properties hold at every $\lambda\in\Lambda$ simultaneously, for all $f$. Fix such a $\lambda\in\Lambda$.

\noindent \emph{Differentiation $\Rightarrow$ shift.} Let $f\in\mathcal{S}(\Hn)$ and $u(s):=B(R_{(se_j,0)}f)(\lambda)$. Since $X_j$ generates $s\mapsto(se_j,0)$, we have
\[
\frac{d}{ds}R_{(se_j,0)}f=R_{(se_j,0)}X_jf=X_jR_{(se_j,0)}f \qquad \text{in } L^1(\Hn).
\]
Since
\[
\lim_{h\to 0}\left\|\frac{R_{((s+h)e_j,0)}f-R_{(se_j,0)}f}{h}-X_jR_{(se_j,0)}f\right\|_{1}=0,
\]
and $B(\cdot)(\lambda):L^1(\Hn)\to\mathcal B(L^2(\mathbb R^n))$ is linear and bounded, applying it to the difference quotient gives
\[
\lim_{h\to 0}\left\|\frac{u(s+h)-u(s)}{h}-B\big(X_jR_{(se_j,0)}f\big)(\lambda)\right\|_{\mathrm{op}}=0,
\]
that is, $u$ is differentiable at $s$ with
\[
u'(s)=B\big(X_jR_{(se_j,0)}f\big)(\lambda).
\]
Applying the first identity in \eqref{diffprop} to $R_{(se_j,0)}f\in\mathcal{S}(\Hn)$ gives
\[
u'(s)\varphi=-i\lambda\,u(s)Q_j\varphi \qquad \text{for } \varphi\in\mathcal{S}(\mathbb{R}^n),
\]
so that $\frac{d}{ds}\big[u(s)e^{i\lambda sQ_j}\varphi\big]=0$. Hence $u(s)\varphi=u(0)e^{-i\lambda sQ_j}\varphi$ for all $\varphi\in\mathcal{S}(\mathbb{R}^n)$, that is,
\[
u(s)=u(0)\,e^{-i\lambda sQ_j}.
\]
This is the first identity in \eqref{shiftprop} for $f\in\mathcal{S}(\Hn)$, and it extends to $f\in L^1(\Hn)$ by density of $\mathcal{S}(\Hn)$ and continuity of $B$. The second identity is obtained in the same way from $Y_j$, with $e^{sP_j}$ in place of $e^{i\lambda sQ_j}$.

\noindent \emph{Shift $\Rightarrow$ differentiation.} Let $f\in\mathcal{S}(\Hn)$ and $\varphi\in\mathcal S(\mathbb R^n)$. Since
\[
\lim_{s\to 0}\left\|\frac{R_{(se_j,0)}f-f}{s}-X_jf\right\|_1=0,
\]
and $B(\cdot)(\lambda)$ is linear and bounded, applying it to the difference quotient gives
\[
\lim_{s\to 0}\left\|\frac{B(R_{(se_j,0)}f)(\lambda)-B(f)(\lambda)}{s}-B(X_jf)(\lambda)\right\|_{\mathrm{op}}=0.
\]
By the first identity in \eqref{shiftprop}, $B(R_{(se_j,0)}f)(\lambda)=B(f)(\lambda)e^{-is\lambda Q_j}$. Applied to $\varphi$, the left-hand difference quotient is therefore
\[
B(f)(\lambda)\,\frac{e^{-is\lambda Q_j}\varphi-\varphi}{s}.
\]
Since $\frac{d}{ds}\big[e^{-is\lambda Q_j}\varphi\big]_{s=0}=-i\lambda Q_j\varphi$, the above expression converges to $B(f)(\lambda)(-i\lambda Q_j\varphi)$ as $s\to 0$. Comparing the two limits on $\varphi$,
\[
B(X_jf)(\lambda)\varphi=-i\lambda B(f)(\lambda)Q_j\varphi.
\]
That is,
\[
B(iX_jf)(\lambda)\varphi=\lambda B(f)(\lambda)Q_j\varphi,
\]
which is the first identity in \eqref{diffprop}. The second identity is obtained in the same way from $Y_j$.

Since $j\in\{1,\dots,n\}$ was arbitrary, this proves
\[
\text{\eqref{diffprop} holds at $\lambda$ for all $f$} \quad\Longleftrightarrow\quad \text{\eqref{shiftprop} holds at $\lambda$ for all $f,s$},
\]
and since $\lambda\in\Lambda$ was arbitrary, $B$ satisfies the Heisenberg differentiation property if and only if it satisfies the Heisenberg shift property. In both directions, the central relation $[X_j,Y_j]=-\partial_t$ on one side (Lemma~\ref{lemB}, Step~1), and the group commutator identity for the one-parameter subgroups on the other (Lemma~\ref{lemG}), recover the same central direction of $\Hn$, as the Lie-algebra and Lie-group forms of the same fact.
\end{proof}

\begin{remark} \label{rem:shift1}
The Heisenberg shift property (Definition~\ref{def:hsp}) is imposed only on the $2n$ one-parameter subgroups generated by $X_1,\dots,X_n,Y_1,\dots,Y_n$, in contrast with the locally compact abelian shift theorem of \cite{Krukowski2}, where the corresponding property is imposed for every $x$ in the group, that is, at infinitely many points. This is a notable instance where passing from an abelian to a non-abelian group weakens the hypothesis needed. A finite covariance hypothesis suffices on $\Hn$ (Lemma~\ref{lemG}) precisely because noncommutativity supplies the missing central direction as a commutator. However, the corresponding hypothesis for a general locally compact abelian group cannot be reduced to a proper subset of the group. Indeed, the commutator-based generation of Lemma~\ref{lemG} is identically trivial on an abelian group. Moreover, $\mathbb{R}^m$ and other connected abelian Lie groups are not finitely generated as abstract groups, so no finite family of one-parameter subgroups can play the role of $\mathcal{K}$.
\end{remark}

\begin{remark}\label{rem:shift2}
The family $\mathcal{K}$ cannot be reduced either. If the pair for any single $j$ is omitted, then the remaining family generates only a proper subgroup of $\Hn$. Since $[X_j,Y_k]=0$ for $k\ne j$, the missing direction cannot be recovered as a commutator of the rest. This will stop Lemma~\ref{lemG}'s construction from reaching $(x_j,y_j)$.
\end{remark}

\begin{remark} \label{rem:diffn}
Krukowski's differentiation theorem (Theorem~\ref{thm:kruk}) is proved for $\mathbb{R}$. An analogous statement and proof hold for $\mathbb{R}^n$, requiring $n$ differentiation identities, one per coordinate direction. By comparison, one might expect the Heisenberg differentiation property on $\Hn$ to require $2n+1$ identities, one for each generator of the $(2n+1)$-dimensional Lie algebra of $\Hn$. It requires instead only $2n$ identities, one each for $X_1,\dots,X_n,Y_1,\dots,Y_n$. This is possible as the behaviour along the remaining, central direction $\partial_t$ is recovered rather than assumed, from the commutator relations satisfied by the left-invariant vector fields on $\Hn$ (Lemma~\ref{lemB}, Step~1). This count cannot be reduced further. The argument of Lemma~\ref{lemB}, Step~3, recovers the dependence of $\rho$ on each pair $(x_j,y_j)$ from the identity for that $j$ only. Moreover, the bracket relations $[X_j,Y_k]=0$ for $j\neq k$ mean no other hypothesis in the family supplies this information. Thus omitting the identity for any single $j$ leaves the dependence on $(x_j,y_j)$ undetermined.

\end{remark}

\begin{remark}
The mechanism of proof of Theorem~\ref{thm:main-diff} departs from the argument for Theorem~\ref{thm:kruk} at exactly the point where noncommutativity enters. On $\mathbb{R}$ the kernel of the unknown transform differs from that of $\mathcal{F}$ by a factor depending only on the dual variable. This is because the exponential kernel is a scalar, hence commutes with any other scalar, so the resulting first-order ODE is straightforward to solve. On $\Hn$, the kernel $\pi_\lambda(z,t)$ is operator-valued, and the corresponding first-order system of PDEs \eqref{Rsys} does not integrate to a constant. It integrates instead to a \emph{left multiple} of $\pi_\lambda(z,t)$ by a fixed operator, and identifying this requires showing that $\rho(z)\pi_\lambda(z)^{-1}$ has vanishing total derivative, which is where the identities of Lemma~\ref{lemA} and the noncommutative structure enter (Lemma~\ref{lemB}, Step~3).
\end{remark}

\section{Concluding remarks}
Theorems~\ref{thm:main-diff} and~\ref{thm:main-shift} together give two independent characterisations of the group Fourier transform on $\Hn$ - by its behaviour with respect to the invariant vector fields, and by its behaviour under right translation. Each requires, as in Krukowski's approach, only that the map $B$ be continuous and linear, together with the Heisenberg Dirac delta property. As shown in Proposition~\ref{prop:equiv}, the two hypotheses are in fact equivalent, being the Lie-algebra and the Lie-group forms of the same underlying condition. This extends Krukowski's shift-property approach, previously available for locally compact abelian groups, and his differentiation-property approach, previously available only for $\mathbb{R}$. It is, to our knowledge, the first instance of this approach for a non-abelian group, without any convolution-based hypotheses.


\begin{thebibliography}{99}

\bibitem{ABM1} S. Achar, M.H.A. Biswas, R. Manna, \textit{Characterization of Fourier transform of Hilbert space valued functions,} arXiv:2610.36957.

\bibitem{ABM2} S. Achar, M.H.A. Biswas, R. Manna, \textit{Characterization of the group Fourier transform of Hilbert space valued functions,} arXiv:2611.03267.

\bibitem{AAM}
S.~Alesker, S.~Artstein-Avidan, and V.~Milman, \emph{A characterization of the Fourier transform and related topics}, C. R. Math. Acad. Sci. Paris \textbf{346} (2008), no.~11-12, 625-628.

\bibitem{Folland}
G.~B.~Folland, \emph{Harmonic Analysis in Phase Space}, Annals of Mathematics Studies, vol.~122, Princeton University Press, Princeton, NJ, 1989.

\bibitem{Geller}
D.~Geller, \emph{Fourier analysis on the Heisenberg group. I. Schwartz space}, J. Funct. Anal. \textbf{36} (1980), no.~2, 205-254.

\bibitem{Jaming}
P.~Jaming, \emph{A characterization of Fourier transforms}, Colloq. Math. \textbf{118} (2010), no.~2, 569-580.

\bibitem{Krukowski1}
M.~Krukowski, \emph{Integral transforms characterized by convolution}, Results Math. \textbf{78} (2023), no.~4, Paper No.~120, 12~pp.

\bibitem{Krukowski2}
M.~Krukowski, \emph{Characterizing the Fourier transform by its properties}, Pacific J. Math. \textbf{329} (2024), no.~2, 217-232.

\bibitem{Ksft}
R.~Lakshmi Lavanya, 
\emph{Some Characterisations
of the Strichartz Fourier Transform on the Heisenberg Group},
preprint, arXiv:2611.01128.

\bibitem{LT2012}
R.~Lakshmi Lavanya and S.~Thangavelu, \emph{A characterisation of the Fourier transform on the Heisenberg group}, Ann. Funct. Anal. \textbf{3} (2012), no.~1, 109-120.

\bibitem{LT2014}
R.~Lakshmi Lavanya and S.~Thangavelu, \emph{Revisiting the Fourier transform on the Heisenberg group}, Publ. Mat. \textbf{58} (2014), no.~1, 47-63.

\bibitem{KS}
N.~Shravan Kumar and S.~Sivananthan, \emph{Characterisation of the Fourier transform on compact groups}, Bull. Aust. Math. Soc. \textbf{93} (2016), no.~3, 467-472.

\bibitem{Thangavelu}
S.~Thangavelu, \emph{Harmonic Analysis on the Heisenberg Group}, Progress in Mathematics, vol.~159, Birkh\"auser, Boston, 1998.

\end{thebibliography}
\end{document}